\documentclass[12pt,reqno]{amsart}

\usepackage{amssymb,amsmath,amsthm,amstext,latexsym,amsfonts}
\usepackage{color}
\usepackage{hyperref}
\hypersetup{
  colorlinks=true,
  linkcolor=magenta,
  citecolor=blue,
}

\usepackage[T1]{fontenc}
\newcommand{\s}[1]{\ensuremath{\boldsymbol{#1}}}

\title[A matrix approach to the Yang multiplication theorem]{A matrix approach to the Yang multiplication theorem}
\author{Akihiro Munemasa}
\address{Research Center for Pure and Applied Mathematics\\
Graduate School of Information Sciences\\
Tohoku University\\ Japan}
\email{munemasa@math.is.tohoku.ac.jp}

\author{Pritta Etriana Putri}
\address{Research Center for Pure and Applied Mathematics\\
Graduate School of Information Sciences\\
Tohoku University \\ Japan\and
Combinatorial Mathematics Research Group\\
Institut Teknologi Bandung\\
Bandung\\ Indonesia}
\email{pritta@ims.is.tohoku.ac.jp}
\dedicatory{Dedicated to the memory of Professor Noboru Ito}

\newtheorem{lem}{Lemma}[section]
\newtheorem{thm}[lem]{Theorem}

\theoremstyle{definition}
\newtheorem{dfn}[lem]{Definition}

\DeclareMathOperator{\seq}{seq}
\newcommand{\R}{\mathcal{R}}
\newcommand{\Z}{\mathbb{Z}}
\newcommand{\T}{\mathcal{T}}
\DeclareMathOperator{\supp}{supp}

\newcommand{\F}[4]{\ensuremath{(\phi_{\s{#1}}\phi^*_{\s{#1}}+\phi_{\s{#2}}\phi^*_{\s{#2}}+\phi_{\s{#3}}\phi^*_{\s{#3}}+\phi_{\s{#4}}\phi^*_{\s{#4}}})}
\newcommand{\G}[4]{\ensuremath{(\psi_{\s{#1}}\psi^*_{\s{#1}}+\psi_{\s{#2}}\psi^*_{\s{#2}}+\psi_{\s{#3}}\psi^*_{\s{#3}}+\psi_{\s{#4}}\psi^*_{\s{#4}}})}

\date{\today }

\begin{document}

\begin{abstract}
In this paper, we use two-variable Laurent polynomials attached to
matrices to encode properties of compositions of sequences. The
Lagrange identity in the ring of Laurent polynomials is then used
to give a short and transparent proof of a theorem about
the Yang multiplication. 
\end{abstract}

\maketitle

\section{Introduction}
Many classes of complementary sequences have been investigated in the literature (see \cite{GC}). A quadruple of $(\pm 1)$-sequences $(\s{a},\s{b},\s{c},\s{d})$ of length $m,m,n,n$, respectively, is called \emph{base sequences} if 
\[
N_{\s{a}}(j)+N_{\s{b}}(j)+N_{\s{c}}(j)+N_{\s{d}}(j)=0
\]
for all positive integers $j$, where 
\[
N_{\s{s}}(j)=
\begin{cases}
\sum_{i=0}^{l-j-1} s_is_{i+j} &\text{if $0\leq j<l$},\\
0 &\text{otherwise},
\end{cases}
\]
for $\s{s}=(s_0,\ldots,s_{l-1})\in\{\pm1\}^l$. We denote by $BS(m,n)$ the set of base sequences of length $m$, $m$, $n$, $n$. If $(\s{a},\s{b},\s{c},\s{d})\in BS(m,n)$, then it is complementary with weight $2(m+n)$. In \cite{CH2}, Yang proved the following theorem, which is known as one version of 
the \emph{Yang multiplication theorem}:

\begin{thm}[\text{\cite[Theorem 4]{CH2}}]\label{Y1}
If $BS(m+1,m)\neq\emptyset$ and $BS(n+1,n)\neq\emptyset$, then $BS(m',m')\neq \emptyset$ with $m'=(2m+1)(2n+1)$.  
\end{thm} 

The well-known Hadamard conjecture states that Hadamard matrices of order $4n$ exist for every positive integer $n$. 
A consequence of Theorem~\ref{Y1} is the existence of a Hadamard
matrix of order $8m'$ for a positive integer $m'$ 
satisfying the hypotheses.
Indeed, a class of sequences called $T$-sequences with length $2m'$ can be obtained from $BS(m',m')$ \cite{DD2}, and Hadamard matrices of order $8m'$ can be produced from $T$-sequences with length $2m'$ by using Goethals--Seidel arrays \cite{T}. For more information on $T$-sequences, we refer the reader to \cite{GC, KH1, KH2, CK1}. 

In order to prove Theorem \ref{Y1}, Yang used the Lagrange identity for polynomial rings. Let $\mathbb{Z}[x^{\pm1}]$ be the ring of Laurent polynomials over $\mathbb{Z}$ and  
 $*: \Z[x^{\pm1}]\rightarrow \Z[x^{\pm1}]$ be the involutive automorphism defined by $x\mapsto x^{-1}$. Let $\s{a}=(a_0,\dots, a_{l-1})\in \Z^l$. We define 
the \emph{Hall polynomial} $\phi_{\s{a}}(x)\in \Z[x^{\pm1 }]$ of
$\s{a}$ by
\begin{align*}
\phi_{\s{a}}(x)&=\sum_{i=0}^{l-1}a_ix^i.
\end{align*}
It is easy to see that a quadraple $(\pm1)$-sequences $(\s{a},\s{b},\s{c},\s{d})$ of length $m,m,n,n$, respectively, is a base sequences if and only if
\[
\F{a}{b}{c}{d}(x)=2(m+n). 
\]
Suppose $(\s{a},\s{b},\s{c},\s{d})\in BS(n+1,n)$ and $(\s{f},\s{g},\s{h},\s{e})\in BS(m+1,m)$. 
The proof of Theorem \ref{Y1} in \cite{CH2} is by establishing the identity 
\begin{align}
&(\phi_{\s{q}}\phi^*_{\s{q}}+\phi_{\s{r}}\phi^*_{\s{r}}+\phi_{\s{s}}\phi^*_{\s{s}}+\phi_{\s{t}}\phi^*_{\s{t}})(x)\notag\\
&=
(\phi_{\s{a}}\phi^*_{\s{a}}+\phi_{\s{b}}\phi^*_{\s{b}}+\phi_{\s{c}}\phi^*_{\s{c}}+\phi_{\s{d}}\phi^*_{\s{d}})(x^2)(\phi_{\s{e}}\phi^*_{\s{e}}+\phi_{\s{f}}\phi^*_{\s{f}}+\phi_{\s{g}}\phi^*_{\s{g}}+\phi_{\s{h}}\phi^*_{\s{h}})(x^{2(2m+1)}),
\label{id}
\end{align}
after defining the sequences $\s{q}$, $\s{r}$, $\s{s}$, $\s{t}$ appropriately such that, in particular, 
\begin{align*}
\phi_{\s{q}}(x)&=\phi_{\s{a}}(x^2)\phi_{\s{f}^*}(x^{2(2m+1)})+x\phi_{\s{c}}(x^2)\phi_{\s{g}}(x^{2(2m+1)})\\
&\quad-x^{2(2m+1)}\phi_{\s{b}^*}(x^2)\phi_{\s{e}}(x^{2(2m+1)})+x^{2(2m+1)+1}\phi_{\s{d}}(x^2)\phi_{\s{h}}(x^{2(2m+1)}).
\end{align*}
A key to the proof is the Lagrange identity (see \cite[Theorem L]{CH2}): given $a$, $b$, $c$, $d$, $e$, $f$, $g$, $h$ in a commutative ring with an involutive automorphism $*$, set 
\begin{equation}\label{LL}
\begin{aligned}
q&=af^*+cg-b^*e+dh,\\
r&=bf^*+dg^*+a^*e-ch^*,\\
s&=ag^*-cf-bh-d^*e,\\
t&=bg-df+ah^*+c^*e.
\end{aligned}
\end{equation}
Then 
\begin{equation}\label{Lag}
qq^*+rr^*+ss^*+tt^*
=(aa^*+bb^*+cc^*+dd^*)(ee^*+ff^*+gg^*+hh^*).
\end{equation}

However, the derivation of (\ref{id}) from (\ref{Lag}) is not so immediate since one has to define $a,b,c,d,e,f, g,h,$ as 
\begin{align*}
&\phi_{\s{a}}(x^2), \phi_{\s{b}}(x^2), x\phi_{\s{c}}(x^2), x\phi_{\s{d}}(x^2),\\
& x^{2m+(1-n)(2m+1)}\phi_{\s{e}}(x^{2(2m+1)}),x^{-n(2m+1)}\phi_{\s{f}}(x^{2(2m+1)}), \\
&x^{-n(2m+1)}\phi_{\s{g}}(x^{2(2m+1)}),x^{(1-n)(2m+1)}\phi_{\s{h}}(x^{2(2m+1)}),
\end{align*}
rather than 
\[
\phi_{\s{a}}(x^2), \phi_{\s{b}}(x^2), \phi_{\s{c}}(x^2), \phi_{\s{d}}(x^2),
\phi_{\s{e}}(x^{2(2m+1)}),\phi_{\s{f}}(x^{2(2m+1)}),
\phi_{\s{g}}(x^{2(2m+1)}),\phi_{\s{h}}(x^{2(2m+1)}),
\]
respectively. 
We note that \DJ okovi\'c and Zhao \cite{DD} observed some connection
between the Yang multiplication theorem and the octonion algebra. 
More information on the
Yang multiplication theorem and constructions of complementary sequences can be found in \cite{CK2}. 

In this paper, we give a more straightforward proof of Theorem \ref{Y1}.
Our approach is by constructing  a matrix $Q$ from the eight sequences  $\s{a},\s{b},\s{c},\s{d},\s{e},\s{f},\s{g},\s{h}$ and produce Laurent polynomials $\psi_{\s{s}}(x)$ for $\s{s}\in \{\s{a},\s{b},\s{c},\s{d},\s{e},\s{f},\s{g},\s{h}\}$ of single variable and a Laurent polynomial $\psi_Q(x,y)$ of two variables for a matrix $Q$, such that  
\[
\psi_Q(x,y)=\psi_{\s{a}}(x)\psi_{\s{f}}(y)+\psi_{\s{c}}(x)\psi_{\s{g}}(y)+
\psi_{\s{b}}(x)\psi_{\s{e}}(y)+\psi_{\s{d}}(x)\psi_{\s{h}}(y).
\] 
This gives an interpretation of the Lagrange identity in term of sequences and matrices, i.e. there exist matrices $Q,R,S,T$ such that
\begin{align*}
&(\psi_Q\psi^*_Q+\psi_R\psi^*_R+\psi_S\psi^*_S+\psi_T\psi^*_T)(x,y)\\
&=(\psi_{\s{a}}\psi^*_{\s{a}}+\psi_{\s{b}}\psi^*_{\s{b}}+\psi_{\s{c}}\psi^*_{\s{c}}+\psi_{\s{d}}\psi^*_{\s{d}})(x)
(\psi_{\s{e}}\psi^*_{\s{e}}+\psi_{\s{f}}\psi^*_{\s{f}}+\psi_{\s{g}}\psi^*_{\s{g}}+\psi_{\s{h}}\psi_{\s{h}}^*)(y).
\end{align*}
Then (\ref{id}) follows immediately by noticing $\psi_{Q}(x,x^{(2n+1)})=\psi_{\s{q}}(x)$ and $(\psi_{\s{a}}\psi^*_{\s{a}})(x)=(\phi_{\s{a}}\phi^*_{\s{a}})(x^2)$. 
 
The paper is organized as follows. In Section 2, we will define a Laurent polynomial $\psi_{\s{a}}(x)$ for a sequence $\s{a}$ and introduce basic properties of $\psi_{\s{a}}(x)$. We will also show how to combine sequences and matrices to produce new sequences and matrices, eventually leading to a construction of a matrix from a given set of eight sequences. Finally, in Section 3, we will prove Theorem \ref{Y1} as a consequence of the Lagrange identity in the ring of Laurent polynomials of two variables. We note here that Theorem \ref{Y1} \cite[Theorem 4]{CH2} is known as one of 
the Yang multiplication theorem. Other versions of 
the Yang multiplication theorem will be investigated in subsequent papers.

\section{Preliminary Results}
Let $\R$ be a commutative ring with identity and let $*$ be an involutive automorphism of $\R$.  Moreover, let $\R[x^{\pm1}]$ be the ring of Laurent polynomials over $\R$ and  
 $*: \R[x^{\pm1}]\rightarrow \R[x^{\pm1}]$ be the extension of the involutive automorphism $*$ of $\R$ defined by $x\mapsto x^{-1}$. 

\begin{dfn}  \label{def:2.5}
Let $\s{a}=(a_0,\dots a_{l-1})\in \R^l$. 
We define 
the \emph{Hall polynomial} $\phi_{\s{a}}(x)\in \R[x^{\pm1 }]$ of
$\s{a}$ by
\begin{align*}
\phi_{\s{a}}(x)&=\sum_{i=0}^{l-1}a_ix^i.
\end{align*}
We define a Laurent polynomial $
\psi_{\s{a}}(x)\in \R[x^{\pm 1}]$ by
\begin{align*}
\psi_{\s{a}}(x)&=x^{1-l}\phi_{\s{a}}(x^2).
\end{align*}
\end{dfn}
 
Hall polynomials have been used not only by Yang, but also others. See \cite{C} and references therein. For a sequence $\s{a}=(a_0,\ldots, a_{l-1})\in \R^l$ of length $l$  we define $\s{a}^*\in \R^l$ by $(a_{l-1}^*, \ldots, a_0^*)$. It follows immediately that $\s{a}^{**}=\s{a}$ for every $\s{a}\in \R^l$. 

\begin{dfn}\label{DF1}
For a sequence $\s{a}=(a_0,\ldots, a_{l-1})$ of length $l$ with entries in $\R$, we define \emph{the non-periodic autocorrelation $N_{\s{a}}$ of $\s{a}$} by 
\[
N_{\s{a}}(j)=
\begin{cases}
\sum_{i=0}^{l-j-1} a_ia_{i+j}^* &\text{if $0\leq j<l$},\\
0 &\text{otherwise}.
\end{cases}
\]
We say that a set of sequences $\{\s{a}_1,\ldots,\s{a}_n\}$ 
not necessarily all of
the same length, is \textit{complementary with weight $w$} if 
\[
\sum_{i=1}^{n}N_{{\s{a}}_i}(j)=
\begin{cases}
w&\text{if $j=0$},\\ 
0&\text{otherwise}.
\end{cases}
\]
\end{dfn}

By Definition~\ref{DF1} with $\R=\Z$, we see that $(\s{a},\s{b},\s{c},\s{d})\in BS(m,n)$ if and only if $\{\s{a},\s{b},\s{c},\s{d}\}$ is complementary with weight $2(m+n)$.  

\begin{lem}\label{lem:psi}
Let $l$ be a positive integer and $\s{a}\in \R^l$. Then 
\[
\psi_{\s{a}^*}(x)=\psi^{*}_{\s{a}}(x).
\]
\end{lem}
\begin{proof}
Straightforward. 
\end{proof}

\begin{lem}\label{L24}
For sequences $\s{a}_1, \ldots, \s{a}_n$ with entries in $\R$, the following are equivalent. 
\begin{enumerate}
\item $\s{a}_1, \ldots, \s{a}_n$ are complementary with weight $w$, 
\item $\sum_{i=1}^n(\phi_{\s{a}_i}\phi^*_{\s{a}_i})(x)=w$, 
\item $\sum_{i=1}^n(\psi_{\s{a}_i}\psi^*_{\s{a}_i})(x)=w$. 
\end{enumerate}
\end{lem}
\begin{proof}
It is straightforward to check that (i) is equivalent to (ii). Equivalence of (ii) and (iii) is clear since for any sequence $\s{a}$, $\phi_{\s{a}}(x^2)\phi^*_{\s{a}}(x^2)=\psi_{\s{a}}(x)\psi^*_{\s{a}}(x)$ from Definition~\ref{def:2.5}. 
\end{proof}

\begin{dfn}
\label{df1}
Let $\s{a}=(a_0,\ldots, a_{l-1})\in\R^l$. Define 
\begin{align*}
\s{a}/0=(a_0,0,a_1,\ldots,0, a_{l-1})\in \R^{2l-1}, \quad
0/\s{a}=(0,a_0,0,\ldots,a_{l-1},0)\in \R^{2l+1}.
\end{align*}
\end{dfn}

\begin{lem}\label{lem1}
For every $\s{a}\in\R^l$, 
\[\psi_{\s{a}/0}(x)=\psi_{0/\s{a}}(x)=\psi_{\s{a}}(x^2).\]
\end{lem}
\begin{proof}
By Definition~\ref{def:2.5} and Definition~\ref{df1} , we have
\begin{align*}
\psi_{\s{a}/0}(x)&=x^{1-(2l-1)}\phi_{\s{a}/0}(x^2)=x^{2-2l}\phi_{\s{a}}(x^4)
=\psi_{\s{a}}(x^2),  \\
\psi_{0/\s{a}}(x)&=x^{1-(2l+1)}\phi_{0/\s{a}}(x^2)=x^{-2l}x^2\phi_{\s{a}}(x^4)
=\psi_{\s{a}}(x^2). 
\end{align*}
\end{proof}

Now, we will define a Laurent polynomial of two variables for arbitrary matrices. 
Let $\R[x^{\pm1},y^{\pm1}]$ be the ring of Laurent polynomials in two variables $x,y$. We define an involutive ring automorphism $*:\R[x^{\pm1},y^{\pm1}]\rightarrow \R[x^{\pm1},y^{\pm1}]$ by $x\mapsto x^{-1}$, $y\mapsto y^{-1}$ and $a\mapsto a^*$ for $a\in \R$. 

\begin{dfn}\label{seq}
For $A\in \R^{m \times n}$, we denote the row vectors of a matrix $A$ by $\s{a}_0, \ldots,\s{a}_{m-1}$. Define 
\[
\seq(A)=(\s{a}_0\mid \s{a}_1\mid\cdots\mid \s{a}_{m-1})\in \R^{mn},\\
\]
where $\mid$ denotes concatenation, and
\[
\psi_A(x,y)=\sum_{i=0}^{m-1}\psi_{\s{a}_i}(x)y^{2i+1-m}. 
\]
\end{dfn}
Clearly, we have $\psi_{A\pm B}(x,y)=\psi_{A}(x,y)\pm\psi_{B}(x,y)$ for every $A,B\in \R^{m\times n}$. Note that we may regard $\R^n$ as $\R^{1\times n}$. So, for every $\s{a}\in \R^n$, we have $\s{a}^t\in \R^{n\times 1}$ where $t$ denotes the transpose of a matrix. 

\begin{lem}\label{star}
Let $\s{f}\in \R^m$ and $\s{a}\in \R^n$. Then 
\[\psi_{\s{f}^t\s{a}}(x,y)=\psi_{\s{a}}(x)\psi_{\s{f}}(y).\]
\end{lem}
\begin{proof}
Let $\s{f}=(f_0,\ldots,f_{m-1})$. Then 
\begin{align*}
\psi_{\s{f}^t\s{a}}(x,y)&=
\sum_{i=0}^{m-1}\psi_{(\s{f}^t\s{a})_i}(x)y^{2i+1-m}\\
&=\sum_{i=0}^{m-1}f_i\psi_{\s{a}}(x)y^{2i+1-m}\\
&=\psi_{\s{a}}(x)\sum_{i=0}^{m-1}f_iy^{2i+1-m}\\
&=\psi_{\s{a}}(x)\psi_{\s{f}}(y).
\end{align*}
\end{proof}

\begin{lem}\label{seqA}
If $A\in \R^{m\times n}$, then
 \[\psi_{\seq(A)}(x)=\psi_A(x,x^n).\]
\end{lem}
\begin{proof}
Let $\s{a}_0,\ldots, \s{a}_{m-1}$ be the row vectors of $A$. Since $\phi_{\seq(A)}(x)=\sum_{i=0}^{m-1}x^{ni}\phi_{\s{a}_i}(x)$, we have 
\begin{align*}
\psi_{\seq(A)}(x)&=x^{1-mn}\phi_{\seq(A)}(x^2)\\
&=x^{1-mn}\sum_{i=0}^{m-1}x^{2ni}\phi_{\s{a}_i}(x^2)\\
&=x^{1-mn}\sum_{i=0}^{m-1}x^{2ni+n-1}\psi_{\s{a}_i}(x)\\
&=\sum_{i=0}^{m-1}x^{n(2i+1-m)}\psi_{\s{a}_i}(x)\\
&=\psi_A(x,x^n). 
\end{align*}
\end{proof}

\section{Main Result}
We will present our result by three steps. The following lemma is essential to describe 
the Yang multiplication theorem by using matrix approach. 
\begin{lem}
\label{lem3}
Let 
\begin{align*}
\s{a},\s{b},\s{c},\s{d}\in \R^{n},\quad
\s{e},\s{f},\s{g},\s{h}\in \R^{m}.
\end{align*}
Set
\begin{align*}
Q&=\s{f}^{*t}\s{a}+\s{g}^t\s{c}-\s{e}^t\s{b}^*+\s{h}^t\s{d},\\
R&=\s{f}^{*t}\s{b}+\s{g}^{*t}\s{d}+\s{e}^t\s{a}^*-\s{h}^{*t}\s{c},\\
S&=\s{g}^{*t}\s{a}-\s{f}^t\s{c}-\s{h}^t\s{b}-\s{e}^t\s{d}^*,\\
T&=\s{g}^{t}\s{b}-\s{f}^t\s{d}+\s{h}^{*t}\s{a}+\s{e}^t\s{c}^*.
\end{align*}
Then
\begin{align*}
&(\psi_Q\psi^*_Q+\psi_R\psi^*_R+\psi_S\psi^*_S+\psi_T\psi^*_T)(x,y)
\\&=(\psi_{\s{a}}\psi^*_{\s{a}}+\psi_{\s{b}}\psi^*_{\s{b}}+\psi_{\s{c}}\psi^*_{\s{c}}+\psi_{\s{d}}\psi^*_{\s{d}})(x)
(\psi_{\s{e}}\psi^*_{\s{e}}+\psi_{\s{f}}\psi^*_{\s{f}}+\psi_{\s{g}}\psi^*_{\s{g}}+\psi_{\s{h}}\psi_{\s{h}}^*)(y).
\end{align*}
\end{lem}
\begin{proof}
By Lemma \ref{lem:psi} and Lemma \ref{star}, we have
\begin{align*}
\psi_Q(x,y)&=\psi_{\s{a}}(x)\psi_{\s{f}}^*(y)+\psi_{\s{c}}(x)\psi_{\s{g}}(y)-\psi_{\s{b}}^*(x)\psi_{\s{e}}(y)+\psi_{\s{d}}(x)\psi_{\s{h}}(y),\\
\psi_R(x,y)&=\psi_{\s{b}}(x)\psi_{\s{f}}^*(y)+\psi_{\s{d}}(x)\psi_{\s{g}}^*(y)+\psi_{\s{a}}^*(x)\psi_{\s{e}}(y)-\psi_{\s{c}}(x)\psi_{\s{h}}^*(y),\\
\psi_S(x,y)&=\psi_{\s{a}}(x)\psi_{\s{g}}^*(y)-\psi_{\s{c}}(x)\psi_{\s{f}}(y)-\psi_{\s{b}}(x)\psi_{\s{h}}(y)-\psi_{\s{d}}^*(x)\psi_{\s{e}}(y),\\
\psi_T(x,y)&=\psi_{\s{b}}(x)\psi_{\s{g}}(y)-\psi_{\s{d}}(x)\psi_{\s{f}}(y)+\psi_{\s{a}}(x)\psi_{\s{h}}^*(y)+\psi_{\s{c}}^*(x)\psi_{\s{e}}(y).
\end{align*}
Thus, by applying the Lagrange identity, the result follows. 
\end{proof}

For the remainder of this section, we fix a multiplicatively closed subset $\T$ of $\R\setminus \{0\}$ satisfying $-1\in \T=\T^*$. Also, we denote $\T_0=\T\cup\{0\}$. Denote by $\supp(\s{a})$ and $\supp(A)$ the set of indices of nonzero entries of a sequence $\s{a}=(a_0,\ldots,a_{l-1})\in\R^l$ and a matrix $A=[a_{ij}]_{0\leq i\leq m-1,0\leq j\leq n-1}\in\R^{m\times n}$, respectively. 
We say that sequences $\s{a},\s{b}$ are \emph{disjoint} if $\supp(\s{a})\cap\supp(\s{b})=\emptyset$. Matrices $A,B$ are also said to be \emph{disjoint} if $\supp(A)\cap\supp(B)=\emptyset$.
\begin{lem}
\label{lem4}
Let $m$ and $n$ be positive integers, 
\begin{align*}
\begin{split}
\s{a},\s{b}&\in \T^{n+1},\\
\s{c},\s{d}&\in \T^{n},\\
\s{f},\s{g}&\in \T^{m+1},\\
\s{h},\s{e}&\in \T^{m}.
\end{split}
\end{align*}
Set
\begin{align*}
\s{a}'=\s{a}/0, \quad \s{b}'=\s{b}/0,\quad \s{c}'=0/\s{c},\quad\s{d}'=0/\s{d},\\
\s{f}'=\s{f}/0, \quad \s{g}'=\s{g}/0,\quad \s{h}'=0/\s{h},\quad\s{e}'=0/\s{e}.
\end{align*}
Write 
\begin{align}\label{Q}
Q&=\s{f}'^{*t}\s{a'}+\s{g}'^t\s{c'}-\s{e}'^t\s{b}'^*+\s{h}'^t\s{d'},\\\label{R}
R&=\s{f}'^{*t}\s{b}'+\s{g}'^{*t}\s{d}'+\s{e}'^t\s{a}'^*-\s{h}'^{*t}\s{c}',\\\label{S}
S&=\s{g}'^{*t}\s{a}'-\s{f}'^t\s{c}'-\s{h}'^t\s{b}'-\s{e}'^t\s{d}'^*,\\\label{T}
T&=\s{g}'^{t}\s{b}'-\s{f}'^t\s{d}'+\s{h}'^{*t}\s{a}'+\s{e}'^t\s{c}'^*.
\end{align}
Then $Q,R,S,T\in \T^{(2m+1)\times (2n+1)}$ satisfy
\begin{align*}
&(\psi_Q\psi^*_Q+\psi_R\psi^*_R+\psi_S\psi^*_S+\psi_T\psi^*_T)(x,y)
\\&=(\psi_{\s{a}}\psi^*_{\s{a}}+\psi_{\s{b}}\psi^*_{\s{b}}+\psi_{\s{c}}\psi^*_{\s{c}}+\psi_{\s{d}}\psi^*_{\s{d}})(x^2)
(\psi_{\s{e}}\psi^*_{\s{e}}+\psi_{\s{f}}\psi^*_{\s{f}}+\psi_{\s{g}}\psi^*_{\s{g}}+\psi_{\s{h}}\psi_{\s{h}}^*)(y^2).
\end{align*}
\end{lem}
\begin{proof}
Notice that $
\s{a}',\s{b}',\s{c}',\s{d}'\in \T_0^{2n+1}$ and $
\s{e}',\s{f}',\s{g}',\s{h}'\in \T_0^{2m+1}$.

Since $\supp(\s{s}'^*)=\supp(\s{s}')$ for every $\s{s}\in\{\s{a},\s{b},\s{c},\s{d},\s{e},\s{f},\s{g},\s{h}\}$ 
and $(\s{s}', \s{t}')$ is disjoint whenever 
\begin{align*}
\s{s}\in \{\s{a},\s{b}\}, \s{t}\in\{\s{c},\s{d}\}\quad\text{or}\quad\s{s}\in \{\s{f},\s{g}\}, \s{t}\in\{\s{h},\s{e}\}, 
\end{align*}
matrices $A$ and $B$ are disjoint whenever $A\neq B$ and 
\[A,B\in\{\s{f}'^{*t}\s{a'},\s{g}'^t\s{c'},\s{e}'^t\s{b}'^*,\s{h}'^t\s{d'}\}. 
\]
Also, 
\begin{align*}
\supp(\s{a}')\cup\supp(\s{c}')&=\supp(\s{b}'^*)\cup\supp(\s{d}')=\{0,\ldots,2n\}, \\
\supp(\s{f}'^*)=\supp(\s{g}'), &\quad \supp(\s{e}')=\supp(\s{h}').
\end{align*}
Hence
\begin{align*}
\supp(Q)&=\supp(\s{f}'^{*t}\s{a'})\cup\supp(\s{g}'^t\s{c'})\cup\supp(\s{e}'^t\s{b}'^*)\cup\supp(\s{h}'^t\s{d'})\\
&=\{(i,j):i\in\supp(\s{g}'), j\in\supp(\s{a}')\cup\supp(\s{c}')\}\\
&\quad\quad\quad\cup\{(i,j):i\in\supp(\s{e}'), j\in\supp(\s{b}'^*)\cup\supp(\s{d}')\}\\
&=\{(i,j):i\in\supp(\s{g}')\cup\supp(\s{e}'), j\in\{0, \ldots, 2n\}\}\\
&=\{0, \ldots, 2m\}\times\{0, \ldots, 2n\}.
\end{align*}
By a similar argument, we obtain 
\[
\supp(R)=\supp(S)=\supp(T)=\{0, \ldots, 2m\}\times\{0, \ldots, 2n\}.
\]
Therefore, $Q,R,S,T\in \T^{(2m+1)\times (2n+1)}$.
The claimed identity follows from Lemma \ref{lem1} and Lemma \ref{lem3}.
\end{proof}

\begin{thm}\label{T3}
Let $m,n$ be positive integers, and suppose
\begin{align*}
\s{a},\s{b}&\in \T^{n+1}, \\
\s{c},\s{d}&\in \T^{n}, \\
\s{f},\s{g}&\in \T^{m+1}, \\
\s{h},\s{e}&\in \T^{m} 
\end{align*}
satisfy
\begin{align*}
\G{a}{b}{c}{d}(x)&=2(2n+1), \\
\G{e}{f}{g}{h}(x)&=2(2m+1).
\end{align*}
Then there exist $\s{q},\s{r},\s{s},\s{t}\in\T^{(2m+1)(2n+1)}$ such that 
\begin{align*}
\G{q}{r}{s}{t}(x)=4(2m+1)(2n+1).
\end{align*}

\end{thm}
\begin{proof}
Define $Q,R,S,T$ as in (\ref{Q}), (\ref{R}), (\ref{S}), (\ref{T}), respectively. Write 
\[
\s{q}=\seq(Q),\quad\s{r}=\seq(R),\quad\s{s}=\seq(S),\quad\s{t}=\seq(T).
\]
By Lemma \ref{lem4}, $\s{q},\s{r},\s{s},\s{t}\in \T^{(2m+1)(2n+1)}$. Applying Lemma \ref{seqA} and Lemma \ref{lem4}, we have
\begin{align*}
&(\psi_{\s{q}}\psi^*_{\s{q}}+\psi_{\s{r}}\psi^*_{\s{r}}+\psi_{\s{s}}\psi^*_{\s{s}}+\psi_{\s{t}}\psi^*_{\s{t}})(x)\\
&=(\psi_Q\psi^*_Q+\psi_R\psi^*_R+\psi_S\psi^*_S+\psi_T\psi^*_T)(x,x^{2n+1})\\
&=(\psi_{\s{a}}\psi^*_{\s{a}}+\psi_{\s{b}}\psi^*_{\s{b}}+\psi_{\s{c}}\psi^*_{\s{c}}+\psi_{\s{d}}\psi^*_{\s{d}})(x^2)
(\psi_{\s{e}}\psi^*_{\s{e}}+\psi_{\s{f}}\psi^*_{\s{f}}+\psi_{\s{g}}\psi^*_{\s{g}}+\psi_{\s{h}}\psi_{\s{h}}^*)(x^{2(2n+1)})\\
&=4(2m+1)(2n+1).
\end{align*}
Hence the proof is complete. 
\end{proof}

Finally, we see that Theorem \ref{Y1} follows from Theorem \ref{T3} by setting $\T=\{\pm 1\}\subseteq \Z$. Hence, our method gives a more transparent proof of Theorem \ref{Y1}. Indeed, by taking $(\s{a},\s{b},\s{c},\s{d})\in BS(n+1,n)$ and $(\s{f},\s{g},\s{h},\s{e})\in BS(m+1,m)$, the hypotheses in Theorem \ref{T3} are satisfied by Lemma \ref{L24}. Then the resulting sequences $(\s{q},\s{r},\s{s},\s{t})$ belong to $BS(m',m')$ by Lemma \ref{L24} where $m'=(2m+1)(2n+1)$.

\subsection*{Acknowledgements}
We would like to thank Robert Craigen for valuable advice.


\begin{thebibliography}{9}
\bibitem{GC}G. Cohen, D. Rubie, J. Seberry, C. Koukouvinos, S. Kounias, and M. Yamada, A survey of base sequences, disjoint complementary sequences and $OD(4t;t,t,t,t)$, 
{\it J. Combin. Math. Combin. Comput.} 5 (1989) 69--103. 
\bibitem{KH1} H. Kharaghani and C. Koukouvinos, Complementary, base and Turyn sequences in: Handbook of Comb. Des. (C.J. Colbourn and J.H. Dinitz., eds.), 2nd Ed., pp. 317--321, Chapman \& Hall/CRC Press, Boca Raton, FL, 2007.
\bibitem{KH2}H. Kharaghani and B. Tayfeh-Rezaie, A Hadamard matrix of order 428, 
{\it J. Combin. Designs} 13 (2005), 435--440.
\bibitem{CK1}C. Koukouvinos and J. Seberry, Addendum to further results on base sequences, disjoint complementary sequences, $OD(4t;t,t,t,t)$ and the excess of Hadamard matrices, 
{\it Congr. Numer.} 82 (1991), 97--103.  
\bibitem{CK2}C. Koukouvinos, S. Kounias, J. Seberry, C.H. Yang and J. Yang, 
Multiplication of sequences with zero autocorrelation, 
{\it Australas. J. Combin.} 10 (1994), 5--15.
\bibitem{C} R. Craigen, W. Gibson and C. Koukouvinos,
An update on primitive ternary complementary pairs,
{\it J. Combin. Theory Ser. A} 114 (2007), 957--963.
\bibitem{DD}D. \v Z. \DJ okovi\'c and K. Zhao, An octonion algebra originating in combinatorics, 
{\it Proc. Amer. Math. Soc.} 138 (2010), 4187--4195.
\bibitem{DD2}D. \v Z. \DJ okovi\'c , Hadamard matrices of small order and Yang conjecture. J. Combin. Designs, 18 (2010), 254--259.
\bibitem{CH1}C. H. Yang, Lagrange identity for polynomials and $\delta$-codes of length $7t$ and $13t$, 
{\it Proc. Amer. Math. Soc.} 88 (1983), 746--750. 
\bibitem{CH3}C. H. Yang, A composition theorem for $\delta$-codes, Proc. Amer. Math. Soc. 89 (1983), 375--378. 
\bibitem{CH2} C. H. Yang, On composition of four-symbol $\delta$-codes and Hadamard matrices, 
{\it Proc. Amer. Math. Soc.} 107 (1989), 763--776. 
\bibitem{T}R. J. Turyn, An infinite class of Williamson matrices, 
{\it J. Combin. Theory Ser. A} 12 (1972), 319--321.
\end{thebibliography}
\end{document}